\documentclass[11pt]{article}

\usepackage[T1]{fontenc}
\usepackage[utf8]{inputenc}
\usepackage{lmodern}
\usepackage[a4paper,margin=3cm]{geometry}
\usepackage[reqno]{amsmath}
\usepackage{amssymb,amsthm,mathtools}
\usepackage{enumitem}
\usepackage{microtype}
\usepackage[hidelinks]{hyperref}

\newtheorem{theorem}{Theorem}[section]
\newtheorem{proposition}[theorem]{Proposition}
\newtheorem{lemma}[theorem]{Lemma}
\newtheorem{corollary}[theorem]{Corollary}

\theoremstyle{definition}
\newtheorem{definition}[theorem]{Definition}

\theoremstyle{remark}
\newtheorem{remark}[theorem]{Remark}

\newcommand{\R}{\mathbb R}
\newcommand{\K}{\mathcal K}
\newcommand{\Lop}{\mathcal L}
\newcommand{\NA}{\mathrm{NA}}
\newcommand{\wk}{\rightharpoonup}
\newcommand{\spann}{\operatorname{span}}
\newcommand{\dBM}{d_{\mathrm{BM}}}

\title{The weak maximizing property, compact perturbations and duality}

\author{%
Domingo Garc\'ia$^{1}$, Manuel Maestre$^{1}$, Dami\'an Pinasco$^{2,3}$,
and Ignacio Zalduendo$^{2}$\\[0.8ex]
\small $^{1}$Departamento de An\'alisis Matem\'atico, Universitat de Val\`encia,\\
\small 46100 Burjassot (Valencia), Spain\\
\small $^{2}$Departamento de Matem\'atica y Estad\'istica, Universidad Torcuato Di Tella,\\
\small Av. Figueroa Alcorta 7350, C1428BCW Buenos Aires, Argentina\\
\small $^{3}$Consejo Nacional de Investigaciones Cient\'ificas y T\'ecnicas (CONICET),\\
\small Buenos Aires, Argentina\\[0.8ex]
\small \texttt{domingo.garcia@uv.es}; \texttt{manuel.maestre@uv.es}\\
\small \texttt{dpinasco@utdt.edu}; \texttt{izalduendo@utdt.edu}\\[0.5ex]
}

\date{}

\begin{document}

\maketitle

\begin{abstract}
We construct a one-parameter family $(X_c)_{0<c<1}$ of real reflexive Banach spaces,
isomorphic to $\ell_2$ and satisfying $\dBM(X_c,\ell_2)\le c^{-1}$, such that
$(X_c,X_c)$ fails the weak maximizing property while $(X_c^*,X_c^*)$ has it. The failure
of the WMP for $X_c$ is witnessed by a diagonal operator which does not attain
its norm but admits a non-weakly null maximizing sequence. On the dual side, a
separated-block estimate yields, after dualization, a reverse Pythagorean inequality
and a quantitative form of the Opial property; together with property $(M)$, this gives
the weak maximizing property for $X_c^*$. We also show that the compact perturbation
property is self-dual for reflexive pairs. Since the weak maximizing property implies
the compact perturbation property, each $(X_c,X_c)$ has the compact perturbation
property. Consequently, the compact perturbation property does not imply the weak
maximizing property even arbitrarily close to Hilbert space, and the weak maximizing
property is not stable under duality in the reflexive setting.
\end{abstract}

\medskip


\medskip


\medskip

\section{Introduction}

Let $E$ and $F$ be Banach spaces and let $T\in\Lop(E,F)$. A sequence
$(x_n)\subset S_E$ is called a \emph{maximizing sequence} for $T$ if
\[
        \|Tx_n\|\longrightarrow \|T\|.
\]
The relation between maximizing sequences and norm attainment already appears in the
work of Pellegrino and Teixeira \cite{PT2009}. They proved, in particular, that for
$1<p<\infty$ and $1\le q<\infty$, an operator from $\ell_p$ into $\ell_q$ attains its
norm if and only if it admits a maximizing sequence which is not weakly null. Motivated
by this phenomenon, Aron, Garc\'ia, Pellegrino and Teixeira introduced in
\cite{AGPT2020} the \emph{weak maximizing property}: a pair $(E,F)$ has the WMP if
every operator $T\in\Lop(E,F)$ admitting a non-weakly null maximizing sequence attains
its norm.

In this introduction we will freely use some concepts and their corresponding notation;
precise definitions can be found in the next section of preliminaries.

One of the first consequences of the WMP observed in \cite[Proposition~2.4]{AGPT2020}
is the following stability under compact perturbations, which extends an earlier result
of Kover for Hilbert spaces \cite{Kover2005}. If $(E,F)$ has the WMP,
$T\in\Lop(E,F)$ and $K\in\K(E,F)$, then
\[
       \|T+K\|>\|T\|
       \quad\Longrightarrow\quad
       T+K\in\NA(E,F).
\]
This condition was later called the \emph{compact perturbation property} (CPP); see
\cite{JMRZ2024}. The converse implication is much less clear. A particularly suggestive
phenomenon was found by Dantas, Jung and Mart\'inez-Cervantes. They showed that if $E$
is an infinite-dimensional reflexive space, then the above compact perturbation
conclusion holds for the pair $(E,c_0)$; nevertheless, after an equivalent renorming of
$E$, the corresponding pair may fail the WMP \cite[Proposition~3.6]{DJM2021}. Since
$c_0$ is not reflexive, this left open the possibility that reflexivity of both spaces
might force the converse. Question~4.4 of \cite{DJM2021}, explicitly attributed there
to Richard Aron, asks precisely whether a reflexive pair with the CPP must have the
WMP.

The WMP has since been studied from several points of view. Dantas, Jung and
Mart\'inez-Cervantes investigated its behavior under subspaces, quotients and direct
sums \cite{DJM2021}. Garc\'ia-Lirola and Petitjean considered maximizing properties for
other topologies, in particular weak$^*$ variants, and obtained criteria in terms of
asymptotic geometry \cite{GLP2021}. Jung, Mart\'inez-Cervantes and Rueda Zoca studied
norm-attaining compact and rank-one perturbations and their relation with the WMP
\cite{JMRZ2024}. More recently, Han and Kim related the WMP and the CPP to property
$(M)$, the Opial property and $M$-ideals of compact operators \cite{HanKim2025}. In a
different direction, Miranda completed the characterization of the pairs
$(L_p[0,1],L_q[0,1])$ having the WMP \cite{Miranda2026}.

A second natural question concerns duality. Garc\'ia-Lirola and Petitjean asked in
\cite[Question~5.12]{GLP2021} whether the WMP of $(E,F)$ implies the
weak$^*$-to-weak$^*$ maximizing property of $(F^*,E^*)$. For reflexive spaces the weak
and weak$^*$ topologies involved coincide, so this becomes a question about the usual
WMP of the dual pair.

There is a simple connection between these two questions. As we show in
Proposition~\ref{prop:CPPdual}, the CPP is self-dual for reflexive pairs. Consequently,
if the CPP implied the WMP for all reflexive pairs, then the WMP itself would be stable
under duality in the reflexive setting. Thus a reflexive counterexample to dual stability
of the WMP automatically yields a negative answer to Aron's question.

We construct such counterexamples as a one-parameter family. For every $0<c<1$ we
define a real reflexive space $X_c$, isomorphic to $\ell_2$, such that
\[
       (X_c,X_c)\text{ does not have the WMP},
       \qquad
       (X_c^*,X_c^*)\text{ has the WMP}.
\]
Moreover,
\[
       \dBM(X_c,\ell_2)=\frac1c,
\]
so the examples approach Hilbert space in Banach--Mazur distance as $c\uparrow1$.
Since WMP implies CPP, $(X_c^*,X_c^*)$ has the CPP; the self-duality of CPP then shows
that $(X_c,X_c)$ has the CPP as well. Hence each $X_c$ is an explicit reflexive
counterexample to CPP implying WMP, and these counterexamples can be chosen
arbitrarily close to $\ell_2$. The point of the quantitative approximation here is the
simultaneous occurrence of CPP and failure of WMP on $(X_c,X_c)$, together with WMP on
$(X_c^*,X_c^*)$, rather than the mere possibility of destroying WMP by a small
renorming of Hilbert space.

The construction uses the spaces $A(N_n)$ introduced by Kalton in his study of
$M$-ideals of compact operators \cite{Kalton1993}. The norms used here depend on the
parameter $c$. Two features of them drive the argument. A first fact is that the norms
become asymptotically flat along vectors of the form
$\sqrt{1-c^2}\,e_0+e_n$; this produces a diagonal operator which does not attain its
norm but has a non-weakly null maximizing sequence. On the dual side, an elementary
one-step inequality yields a Pythagorean type estimate for separated blocks. After
dualization this becomes a reverse Pythagorean inequality in $X_c^*$, which gives a
quantitative form of the Opial property. Together with property $(M)$ and the results of
Han and Kim, this yields the WMP of $(X_c^*,X_c^*)$.

The paper is organized as follows. In Section~2 we recall the WMP and the CPP, show
the self-duality of the CPP for reflexive pairs, and state the criterion involving
property $(M)$ and the Opial property that will be used on the dual side. In Section~3
we introduce the family $X_c$, prove its quantitative proximity to $\ell_2$, and
establish the two estimates needed later. Section~4 constructs the diagonal operator
showing that $(X_c,X_c)$ fails the WMP. Section~5 identifies $X_c^*$, proves a
quantitative Opial inequality and deduces the WMP for $(X_c^*,X_c^*)$. The final
corollaries answer the two questions above, and we close with a remark describing a
direct proof of the CPP for $X_c$ through Kalton's $M$-ideal criterion.

\section{Preliminaries}

All Banach spaces in the paper are real. We write $\Lop(E,F)$ and $\K(E,F)$ for the
spaces of bounded and compact linear operators from $E$ to $F$, respectively, and
$\NA(E,F)$ for the set of norm-attaining operators. We write $x_n\rightharpoonup x$
to denote weak convergence.

\begin{definition}
A pair $(E,F)$ has the \emph{weak maximizing property} (WMP) if every
$T\in\Lop(E,F)$ which admits a maximizing sequence that is not weakly null attains its
norm.
\end{definition}

\begin{definition}
A pair $(E,F)$ has the \emph{compact perturbation property} (CPP) if for every
$T\in\Lop(E,F)$ and $K\in\K(E,F)$,
\[
       \|T+K\|>\|T\|
       \quad\Longrightarrow\quad
       T+K\in\NA(E,F).
\]
\end{definition}

As recalled in the introduction, WMP implies CPP
\cite[Proposition~2.4]{AGPT2020}. The next elementary observation will allow us to use
this implication in the opposite direction of duality.

\begin{proposition}\label{prop:CPPdual}
Let $E$ and $F$ be reflexive Banach spaces. Then
\[
       (E,F)\text{ has the CPP}
       \quad\Longleftrightarrow\quad
       (F^*,E^*)\text{ has the CPP}.
\]
\end{proposition}

\begin{proof}
Assume first that $(E,F)$ has the CPP. Let $J_E:E\to E^{**}$ and
$J_F:F\to F^{**}$ denote the canonical isometries. Given
$S\in\Lop(F^*,E^*)$ and $L\in\K(F^*,E^*)$, define
\[
       T=J_F^{-1}S^*J_E\in\Lop(E,F),
       \qquad
       K=J_F^{-1}L^*J_E\in\Lop(E,F).
\]
Reflexivity makes these formulas well defined, and the canonical identifications give
$S=T^*$ and $L=K^*$. Moreover, $K$ is compact by Schauder's theorem. If
\[
       \|S+L\|>\|S\|,
\]
then, since adjoints preserve the operator norm,
\[
       \|T+K\|>\|T\|.
\]
Hence $T+K$ attains its norm. Norm attainment passes from an operator to its adjoint, so
$S+L=(T+K)^*$ attains its norm. Therefore $(F^*,E^*)$ has the CPP. The converse follows
by applying the same argument to the reflexive spaces $F^*$ and $E^*$.
\end{proof}

\begin{remark}\label{rem:logical-link}
Proposition~\ref{prop:CPPdual} links the two questions discussed in the introduction. If
one had
\[
       \text{CPP}\Longrightarrow\text{WMP}
\]
for every reflexive pair, then for reflexive $E,F$ one would have
\[
 (E,F)\text{ WMP}
 \Longrightarrow (E,F)\text{ CPP}
 \Longrightarrow (F^*,E^*)\text{ CPP}
 \Longrightarrow (F^*,E^*)\text{ WMP}.
\]
Thus failure of dual stability for WMP necessarily produces a reflexive pair with CPP
but without WMP.
\end{remark}

We next recall property $(M)$ in the form needed below.

\begin{definition}
A Banach space $E$ has \emph{property $(M)$} if, whenever $x,y\in E$ satisfy
$\|x\|=\|y\|$ and $(u_n)$ is weakly null in $E$, one has
\[
     \limsup_n\|x+u_n\|=\limsup_n\|y+u_n\|.
\]
\end{definition}

We also recall the Opial property \cite{Opial1967}. A Banach space $E$ has the
\emph{Opial property} if for every $x\ne0$ and every weakly null sequence $(x_n)$,
\[
       \limsup_n\|x_n\|<\limsup_n\|x+x_n\|.
\]
Han and Kim introduced strict property $(M)$ and showed that property $(M)$ together
with the corresponding property $(O)$ is equivalent to strict property $(M)$
\cite[Lemma~14]{HanKim2025}. They also proved that the diagonal pair $(E,E)$ has the
relevant property precisely when the space $E$ has it
\cite[Proposition~15]{HanKim2025}, and that strict property $(M)$ implies the WMP when
the domain is reflexive \cite[Theorem~12]{HanKim2025}. We shall use the following
immediate consequence.

\begin{proposition}\label{prop:HanKimcriterion}
Let $E$ be reflexive. If $E$ has property $(M)$ and the Opial property, then $(E,E)$ has
the WMP.
\end{proposition}

\section{The family \texorpdfstring{$X_c$}{Xc} and its basic estimates}

A norm $N$ on $\R^2$ is \emph{absolute} if
$N(s,t)=N(|s|,|t|)$ and \emph{normalized} if
$N(1,0)=N(0,1)=1$. Given a sequence $(N_n)_{n\ge1}$ of absolute normalized norms on
$\R^2$, we shall define a norm on $c_{00}$. To this end, define recursively a family of
norms $M_n$ on $\R^{n+1}$ by
\[
 M_0(x_0)=|x_0|,
\]
and, for $n\ge1$,
\[
 M_n(x_0,\ldots,x_n)
 =N_n\bigl(M_{n-1}(x_0,\ldots,x_{n-1}),|x_n|\bigr).
 \tag{3.1}\label{eq:recursive}
\]
Since each $N_n$ is normalized and absolute, $N_n(s,0)=|s|$, and therefore the norms
$M_n$ are compatible in the sense that
\[
 M_n(x_0,\ldots,x_{n-1},0)=M_{n-1}(x_0,\ldots,x_{n-1}).
\]
Thus, if $x=(x_j)_{j\ge0}\in c_{00}$, we may define
\[
 \|x\|=M_n(x_0,\ldots,x_n)
\]
for any $n$ such that $x_j=0$ for every $j>n$; the compatibility above shows that this
definition is independent of the choice of $n$. The completion of $c_{00}$ with respect
to this norm is denoted by $A(N_n)$. We write $(e_j)_{j\ge0}$ for its canonical basis
and $P_m$ for the canonical projection onto $\spann\{e_0,\ldots,e_m\}$.

Fix $0<c<1$. For $n\ge1$ define
\[
 N_{n,c}(s,t)
 :=\left[
 (1-c^{2n})|t|^{2n}
 +\left(s^2+c^2t^2\right)^n
 \right]^{1/(2n)}.
 \tag{3.2}\label{eq:Nnc}
\]
We set
\[
        X_c:=A(N_{n,c}).
\]

Let
\[
 \alpha_{n,c}:=(1-c^{2n})^{1/(2n)},
 \qquad
 A_c(s,t):=(s^2+c^2t^2)^{1/2}.
\]
Then
\[
 N_{n,c}(s,t)
 =\bigl\|(\alpha_{n,c}|t|,A_c(s,t))\bigr\|_{\ell_{2n}^2}.
 \tag{3.3}\label{eq:Nrepresentation}
\]
Minkowski's inequality, first for $A_c$ and then in $\ell_{2n}^2$, shows that
$N_{n,c}$ is an absolute norm. Moreover,
\[
 N_{n,c}(1,0)=1,
 \qquad
 N_{n,c}(0,1)^{2n}=1-c^{2n}+c^{2n}=1,
\]
so it is normalized. We may therefore use Kalton's results. In particular, the canonical
basis of $X_c$ is $1$-unconditional and $X_c$ has property $(M)$
\cite[Proposition~3.2]{Kalton1993}.

The parameter $c$ controls quantitatively how far the norm is from the Hilbert norm.
We use the convention
\[
 d_{\mathrm{BM}}(E,F)
 =\inf\{\|U\|\,\|U^{-1}\|:U:E\to F\text{ is an isomorphism}\}.
\]

\begin{lemma}\label{lem:l2estimate}
For every $n\ge1$ and $s,t\in\R$,
\[
  (s^2+c^2t^2)^{1/2}
  \le N_{n,c}(s,t)
  \le (s^2+t^2)^{1/2}.
  \tag{3.4}\label{eq:onestep}
\]
Consequently, for every $x=(x_j)\in c_{00}$,
\[
 \left(|x_0|^2+c^2\sum_{j\ge1}|x_j|^2\right)^{1/2}
 \le \|x\|_{X_c}
 \le \|x\|_2.
 \tag{3.5}\label{eq:l2weighted}
\]
In particular,
\[
       c\|x\|_2\le \|x\|_{X_c}\le\|x\|_2,
       \tag{3.6}\label{eq:l2equiv}
\]
so $X_c$ is isomorphic to $\ell_2$, is reflexive, and
\[
       \dBM(X_c,\ell_2)\le \frac1c.
       \tag{3.7}\label{eq:BM}
\]
The same bound holds for $X_c^*$.
\end{lemma}

\begin{proof}
The lower estimate in \eqref{eq:onestep} follows immediately from
\eqref{eq:Nnc}. For the upper estimate it is enough, by absoluteness, to assume
$s,t\ge0$. For $n=1$ the assertion is immediate. If $n>1$ and $t>0$, put
$u=s^2/t^2$ and consider
\[
 g(u):=(u+1)^n-(u+c^2)^n-(1-c^{2n}).
\]
Then $g(0)=0$ and, since $0<c<1$,
\[
 g'(u)=n\bigl((u+1)^{n-1}-(u+c^2)^{n-1}\bigr)>0.
\]
Hence
\[
 (1-c^{2n})t^{2n}+(s^2+c^2t^2)^n\le(s^2+t^2)^n,
\]
which is the desired upper estimate. The case $t=0$ is immediate.

Iterating the two inequalities in the recursion \eqref{eq:recursive} gives
\eqref{eq:l2weighted}, and \eqref{eq:l2equiv} follows. The identity between the
underlying vector spaces therefore gives \eqref{eq:BM}. Passing to adjoints preserves
the distortion of an isomorphism, so the same estimate holds for $X_c^*$.
\end{proof}

\begin{remark}\label{rem:exactBM}
The estimate \eqref{eq:BM} is sharp. More precisely,
\[
       d_{\mathrm{BM}}(X_c,\ell_2)=\frac1c.
\]
Here is a short argument. For fixed $k\ge1$ put
\[
       v_{m,k}=\sum_{j=1}^k e_{m+j}.
\]
Since, for fixed $s\ge0$,
\[
       N_{n,c}(s,1)\longrightarrow
       \max\{1,(s^2+c^2)^{1/2}\}
       \qquad(n\to\infty),
\]
and $|N_{n,c}(s,1)-N_{n,c}(s',1)|\le |s-s'|$, an induction on the
fixed number $k$ of coordinates gives
\[
       \lim_{m\to\infty}\|v_{m,k}\|_{X_c}^2=1+(k-1)c^2.
       \tag{3.8}\label{eq:blockBMlimit}
\]
Let $U:X_c\to\ell_2$ be any isomorphism and write
$A=\|U\|$, $B=\|U^{-1}\|$. Averaging over independent signs
$\varepsilon_j\in\{-1,1\}$ and using the $1$-unconditionality of the canonical basis,
\[
 \frac{k}{B^2}
 \le \sum_{j=1}^k\|Ue_{m+j}\|_2^2
 =\mathbb E_\varepsilon
   \left\|U\sum_{j=1}^k\varepsilon_je_{m+j}\right\|_2^2
 \le A^2\|v_{m,k}\|_{X_c}^2.
\]
Letting first $m\to\infty$ and then $k\to\infty$ yields
$A^2B^2\ge c^{-2}$. Thus every isomorphism from $X_c$ onto $\ell_2$ has distortion at
least $1/c$, and \eqref{eq:BM} gives the equality. By passing to adjoints, the same
argument also gives $d_{\mathrm{BM}}(X_c^*,\ell_2)=1/c$.
\end{remark}

Since the canonical basis is a Schauder basis and $X_c$ is reflexive, it is shrinking.
Thus
\[
       P_m^*f\longrightarrow f\quad(f\in X_c^*),
       \qquad e_n\wk0.
       \tag{3.9}\label{eq:shrinking}
\]
We shall also use that every coordinate projection, in particular $P_m$ and $I-P_m$, is
a contraction; this follows from the $1$-unconditionality of the basis.

We now show the estimate which will drive the analysis of the dual space.

\begin{lemma}\label{lem:Ndifference}
For every $n\ge1$, $t\in\R$ and $0\le b\le a$,
\[
       N_{n,c}(a,t)^2-N_{n,c}(b,t)^2\le a^2-b^2.
       \tag{3.10}\label{eq:Ndifference}
\]
\end{lemma}

\begin{proof}
Fix $n$ and $t$. If $t=0$, then $N_{n,c}(s,0)=|s|$ and the conclusion is immediate.
Assume therefore that $t\ne0$ and, for $u\ge0$, put
\[
 F(u):=N_{n,c}(\sqrt u,t)^2
 =\left[(1-c^{2n})|t|^{2n}
 +(u+c^2t^2)^n\right]^{1/n}.
\]
Then
\[
 F'(u)=
 \frac{(u+c^2t^2)^{n-1}}
 {\left[(1-c^{2n})|t|^{2n}+(u+c^2t^2)^n\right]^{(n-1)/n}}
 \le1.
\]
Hence $u\mapsto F(u)-u$ is non-increasing. Taking $u=a^2$ and $u=b^2$ gives
\eqref{eq:Ndifference}.
\end{proof}

\begin{proposition}\label{prop:blockprimal}
For every $m\ge0$, $u\in P_mX_c$ and $z\in(I-P_m)X_c$,
\[
       \|u+z\|_{X_c}^2\le \|u\|_{X_c}^2+\|z\|_{X_c}^2.
       \tag{3.11}\label{eq:blockprimal}
\]
\end{proposition}

\begin{proof}
Assume first that $z$ is supported in $\{m+1,\ldots,N\}$. Set
\[
 r_m=\|u\|_{X_c},\qquad s_m=0,
\]
and, for $m<j\le N$,
\[
 r_j=N_{j,c}(r_{j-1},|z_j|),
 \qquad
 s_j=N_{j,c}(s_{j-1},|z_j|).
\]
Absolute norms are coordinatewise monotone, and therefore $r_j\ge s_j$. Applying
Lemma~\ref{lem:Ndifference} at each step gives
\[
 r_j^2-s_j^2\le r_{j-1}^2-s_{j-1}^2.
\]
After iteration,
\[
 \|u+z\|_{X_c}^2-\|z\|_{X_c}^2=r_N^2-s_N^2
 \le r_m^2=\|u\|_{X_c}^2.
\]
For arbitrary $z\in(I-P_m)X_c$, apply the finite-support estimate to $P_Nz$ and let
$N\to\infty$.
\end{proof}

The other special feature of $N_{n,c}$ is an asymptotic flattening along a fixed nonzero
direction. For $0<c<1$, put
\[
       \delta_c:=\sqrt{1-c^2}.
\]

\begin{lemma}\label{lem:flattening}
If
\[
       q_{n,c}:=\|\delta_c e_0+e_n\|_{X_c},
\]
then
\[
       q_{n,c}=(2-c^{2n})^{1/(2n)}>1,
       \qquad q_{n,c}\longrightarrow1.
       \tag{3.12}\label{eq:qnc}
\]
Moreover,
\[
       u_{n,c}:=\frac{\delta_c e_0+e_n}{q_{n,c}}\in S_{X_c},
       \qquad
       u_{n,c}\wk \delta_c e_0\ne0.
       \tag{3.13}\label{eq:uncweak}
\]
\end{lemma}

\begin{proof}
All coordinates between $e_0$ and $e_n$ are zero, so the recursion remains equal to
$\delta_c$ until the $n$-th step. Since $\delta_c^2+c^2=1$,
\[
 q_{n,c}=N_{n,c}(\delta_c,1)
 =\left[(1-c^{2n})+1\right]^{1/(2n)}
 =(2-c^{2n})^{1/(2n)}.
\]
The stated properties of $q_{n,c}$ follow immediately. Finally,
\eqref{eq:shrinking} gives $e_n\wk0$, and \eqref{eq:uncweak} follows from
$q_{n,c}\to1$.
\end{proof}

\section{Failure of the weak maximizing property in \texorpdfstring{$X_c$}{Xc}}

We now construct an operator which witnesses the failure of the WMP for every member of
the family. Define $D_c\in\Lop(X_c)$ on the canonical basis by
\[
       D_ce_0=0,
       \qquad
       D_ce_n=d_ne_n,
       \qquad
       d_n:=\frac{n}{n+1}\quad(n\ge1).
       \tag{4.1}\label{eq:Dc}
\]
The formula is independent of $c$; the subscript only indicates the ambient norm. Since
the basis is $1$-unconditional and $0\le d_n<1$, $D_c$ is a contraction. On the other
hand,
\[
       \|D_ce_n\|_{X_c}=d_n\longrightarrow1,
\]
so $\|D_c\|=1$.

The fact that every diagonal coefficient is strictly smaller than one does not by itself
exclude norm attainment on a vector with infinite support. We therefore keep the
following argument explicit.

\begin{proposition}\label{prop:DnotNA}
For every $x=\sum_{j\ge0}x_je_j\in X_c\setminus\{0\}$,
\[
       \|D_cx\|_{X_c}<\|x\|_{X_c}.
\]
In particular, $D_c$ does not attain its norm.
\end{proposition}

\begin{proof}
If $D_cx=0$ there is nothing to prove. Assume $D_cx\ne0$ and choose
$f\in S_{X_c^*}$ such that
\[
       f(D_cx)=\|D_cx\|_{X_c}>0.
\]
For each $j$ choose $\varepsilon_j\in\{-1,1\}$ so that
$\varepsilon_jf(e_j)x_j=|f(e_j)x_j|$, and put
\[
       x^\sharp=\sum_{j\ge0}\varepsilon_jx_je_j.
\]
The $1$-unconditionality of the basis gives
\[
       \|x^\sharp\|_{X_c}=\|x\|_{X_c},
       \qquad
       \|D_cx^\sharp\|_{X_c}=\|D_cx\|_{X_c}.
\]
For every $N$,
\[
\begin{aligned}
 f(P_ND_cx^\sharp)
 &=\sum_{j=1}^N d_j|f(e_j)x_j|\\
 &\ge \left|\sum_{j=1}^N d_jf(e_j)x_j\right|
 =|f(P_ND_cx)|.
\end{aligned}
\]
Passing to the limit and using convergence of the partial sums,
\[
       f(D_cx^\sharp)\ge |f(D_cx)|=\|D_cx\|_{X_c}.
\]
The reverse inequality follows from $\|f\|=1$, and hence
\[
       f(D_cx^\sharp)=\|D_cx\|_{X_c}.
       \tag{4.2}\label{eq:fnormsD}
\]
Since this number is positive, there is some $j_0\ge1$ such that
$|f(e_{j_0})x_{j_0}|>0$. For $N\ge j_0$,
\[
\begin{aligned}
 f(P_Nx^\sharp)-f(P_ND_cx^\sharp)
 &=|f(e_0)x_0|+
   \sum_{j=1}^N(1-d_j)|f(e_j)x_j|\\
 &\ge (1-d_{j_0})|f(e_{j_0})x_{j_0}|>0.
\end{aligned}
\]
Letting $N\to\infty$ and using \eqref{eq:fnormsD}, we obtain
\[
       \|x\|_{X_c}=\|x^\sharp\|_{X_c}
       \ge f(x^\sharp)
       >f(D_cx^\sharp)=\|D_cx\|_{X_c}.
\]
\end{proof}

\begin{theorem}\label{thm:XcnotWMP}
For every $0<c<1$, the pair $(X_c,X_c)$ does not have the weak maximizing property.
\end{theorem}

\begin{proof}
Proposition~\ref{prop:DnotNA} shows that $D_c$ does not attain its norm. On the other
hand, for the vectors $u_{n,c}$ of Lemma~\ref{lem:flattening},
\[
       D_cu_{n,c}=\frac{d_n}{q_{n,c}}e_n,
\]
and therefore
\[
       \|D_cu_{n,c}\|_{X_c}
       =\frac{d_n}{q_{n,c}}
       \longrightarrow1=\|D_c\|.
\]
Thus $(u_{n,c})$ is a maximizing sequence for $D_c$. By
\eqref{eq:uncweak},
\[
       u_{n,c}\wk\delta_c e_0\ne0,
\]
so this maximizing sequence is not weakly null. Hence $(X_c,X_c)$ fails the WMP.
\end{proof}

\begin{remark}\label{rem:XnotOpial}
The same computation shows that $X_c$ does not have the Opial property. Indeed,
$e_n\wk0$, $\|e_n\|_{X_c}=1$, and
\[
       \|\delta_c e_0+e_n\|_{X_c}=q_{n,c}\longrightarrow1.
\]
Thus $X_c$ has property $(M)$ but lacks the strict asymptotic inequality which will
reappear in its dual.
\end{remark}

\section{The dual spaces and the weak maximizing property}

Fix $0<c<1$. If $N$ is an absolute normalized norm on $\R^2$, denote its dual norm by
\[
       N^*(a,b)=\sup\{|as+bt|:N(s,t)\le1\}.
\]
The norm $N^*$ is again absolute and normalized. We shall use the elementary identity
\[
       (E\oplus_N\R)^*=E^*\oplus_{N^*}\R.
       \tag{5.1}\label{eq:dualsum}
\]
Under the natural identification, the norm of $(f,a)$ is $N^*(\|f\|,|a|)$.

\begin{proposition}\label{prop:Xdual}
If $(e_j^*)_{j\ge0}$ is the biorthogonal basis, then, isometrically,
\[
       X_c^*=A(N_{n,c}^*).
       \tag{5.2}\label{eq:Xdual}
\]
In particular, $X_c^*$ has property $(M)$.
\end{proposition}

\begin{proof}
The initial block $E_m=P_mX_c$ is obtained recursively as
\[
       E_m=E_{m-1}\oplus_{N_{m,c}}\R e_m.
\]
Applying \eqref{eq:dualsum} inductively shows that if
$f=\sum_{j=0}^m a_je_j^*$, then its norm is computed by
\[
       r_0=|a_0|,
       \qquad
       r_j=N_{j,c}^*(r_{j-1},|a_j|),
       \qquad
       \|f\|=r_m.
       \tag{5.3}\label{eq:dualrecursion}
\]
To pass from $E_m^*$ to $X_c^*$, note that $P_m$ is a norm-one projection. Hence the
map $E_m^*\to X_c^*$ given by $\varphi\mapsto\varphi\circ P_m$ is an isometry onto
$\spann\{e_0^*,\ldots,e_m^*\}$. Since the basis of $X_c$ is shrinking, these
finite-support functionals are dense in $X_c^*$. This proves \eqref{eq:Xdual}. Since
every $N_{n,c}^*$ is absolute and normalized, property $(M)$ follows from
\cite[Proposition~3.2]{Kalton1993}.
\end{proof}

The block estimate of Proposition~\ref{prop:blockprimal} becomes, after dualization, a
reverse Pythagorean inequality.

\begin{proposition}\label{prop:blockdual}
Let $m\ge0$ and let $f,g\in X_c^*$ satisfy
\[
       P_m^*f=f,
       \qquad
       P_m^*g=0.
\]
Then
\[
       \|f+g\|_{X_c^*}^2
       \ge \|f\|_{X_c^*}^2+\|g\|_{X_c^*}^2.
       \tag{5.4}\label{eq:blockdual}
\]
\end{proposition}

\begin{proof}
Put
\[
       E_m=P_mX_c,
       \qquad
       Z_m=(I-P_m)X_c,
\]
with their inherited norms. Proposition~\ref{prop:blockprimal} says that
\[
       J_m:E_m\oplus_2 Z_m\longrightarrow X_c,
       \qquad
       J_m(u,z)=u+z,
\]
is a contraction. Hence so is
\[
       J_m^*:X_c^*\longrightarrow E_m^*\oplus_2 Z_m^*.
\]
Since $P_m$ and $I-P_m$ are contractions and the assumptions imply that $f$ vanishes on
$Z_m$ and $g$ vanishes on $E_m$, respectively,
\[
       \|f|_{E_m}\|_{E_m^*}=\|f\|_{X_c^*},
       \qquad
       \|g|_{Z_m}\|_{Z_m^*}=\|g\|_{X_c^*}.
\]
Moreover,
\[
       J_m^*(f+g)=(f|_{E_m},g|_{Z_m}).
\]
Therefore
\[
 \bigl(\|f\|_{X_c^*}^2+\|g\|_{X_c^*}^2\bigr)^{1/2}
 =\|J_m^*(f+g)\|
 \le\|f+g\|_{X_c^*},
\]
which is \eqref{eq:blockdual}.
\end{proof}

We now obtain the asymptotic inequality which distinguishes $X_c^*$ from $X_c$.

\begin{theorem}\label{thm:quantOpial}
Let $f\in X_c^*$ and let $(h_k)\subset X_c^*$ be weakly null. If
\[
       L:=\limsup_k\|h_k\|_{X_c^*},
\]
then
\[
       \limsup_k\|f+h_k\|_{X_c^*}
       \ge
       \bigl(\|f\|_{X_c^*}^2+L^2\bigr)^{1/2}.
       \tag{5.5}\label{eq:quantOpial}
\]
In particular, $X_c^*$ has the Opial property.
\end{theorem}

\begin{proof}
Set $a=\|f\|_{X_c^*}$. Since $\|e_0^*\|=1$ and $X_c^*$ has property $(M)$,
\[
       \limsup_k\|f+h_k\|_{X_c^*}
       =\limsup_k\|ae_0^*+h_k\|_{X_c^*}.
       \tag{5.6}\label{eq:Mreplace}
\]
Write
\[
       \alpha_k=h_k(e_0),
       \qquad
       g_k=h_k-\alpha_ke_0^*.
\]
Weak convergence gives $\alpha_k\to0$, and therefore
\[
       \|h_k-g_k\|_{X_c^*}\to0,
       \qquad
       \limsup_k\|g_k\|_{X_c^*}=L.
       \tag{5.7}\label{eq:gklimsup}
\]
Moreover, $P_0^*g_k=0$. Proposition~\ref{prop:blockdual} gives
\[
       \|ae_0^*+g_k\|_{X_c^*}^2
       \ge a^2+\|g_k\|_{X_c^*}^2.
\]
Using \eqref{eq:Mreplace}--\eqref{eq:gklimsup} and the fact that replacing $h_k$ by
$g_k$ changes the relevant norms by a quantity tending to zero, we obtain
\eqref{eq:quantOpial}. If $f\ne0$, the right-hand side is strictly larger than $L$,
and hence $X_c^*$ has the Opial property.
\end{proof}

\begin{theorem}\label{thm:dualWMP}
For every $0<c<1$, the pair $(X_c^*,X_c^*)$ has the weak maximizing property.
\end{theorem}

\begin{proof}
The space $X_c^*$ is reflexive. By Proposition~\ref{prop:Xdual}, it has property $(M)$,
and by Theorem~\ref{thm:quantOpial}, it has the Opial property. The conclusion follows
from Proposition~\ref{prop:HanKimcriterion}.
\end{proof}

We can now draw the two consequences announced in the introduction.

\begin{corollary}\label{cor:duality}
The weak maximizing property is not stable under duality, even among real reflexive
spaces which can be chosen arbitrarily close to $\ell_2$ in Banach--Mazur distance.
More precisely, for every $0<c<1$, if $E_c=X_c^*$, then
\[
       (E_c,E_c)\text{ has the WMP},
       \qquad
       (E_c^*,E_c^*)\text{ does not have the WMP},
\]
and
\[
       \dBM(E_c,\ell_2)\le c^{-1}.
\]
\end{corollary}

\begin{proof}
Theorem~\ref{thm:dualWMP} gives the WMP for $(E_c,E_c)$. Since $X_c$ is reflexive,
$E_c^*=X_c^{**}$ is canonically isometric to $X_c$, and
Theorem~\ref{thm:XcnotWMP} shows that $(E_c^*,E_c^*)$ fails the WMP. The distance
estimate follows from Lemma~\ref{lem:l2estimate}. For reflexive spaces the weak$^*$ and
weak topologies on the dual coincide, so this also gives a negative answer to
\cite[Question~5.12]{GLP2021}.
\end{proof}

\begin{corollary}\label{cor:CPPnotWMP}
For every $0<c<1$, the pair $(X_c,X_c)$ has the compact perturbation property but does
not have the weak maximizing property. In particular, the CPP does not imply the WMP
even for reflexive self-pairs on spaces isomorphic to $\ell_2$, and the counterexamples
may be chosen arbitrarily close to $\ell_2$ in Banach--Mazur distance.
\end{corollary}

\begin{proof}
By Theorem~\ref{thm:dualWMP}, $(X_c^*,X_c^*)$ has the WMP and therefore has the CPP.
Proposition~\ref{prop:CPPdual} then implies that $(X_c,X_c)$ has the CPP. On the other
hand, Theorem~\ref{thm:XcnotWMP} shows that it does not have the WMP. Finally,
\eqref{eq:BM} gives $\dBM(X_c,\ell_2)\le c^{-1}\to1$ as $c\uparrow1$.
\end{proof}

\begin{remark}\label{rem:directCPP}
There is also a direct way to verify the CPP for $(X_c,X_c)$. Kalton's results imply
that $X_c=A(N_{n,c})$ has property $(M)$. If $(P_m)$ denotes the canonical partial sum
projections, then $P_mx\to x$ for every $x\in X_c$ and, since $X_c$ is reflexive and its
canonical basis is shrinking, $P_m^*f\to f$ for every $f\in X_c^*$. Moreover, the
$1$-unconditionality of the basis gives
\[
       \|I-2P_m\|=1\qquad(m\ge0).
\]
Thus Kalton's criterion \cite[Theorem~2.4]{Kalton1993} shows that $\K(X_c)$ is an
$M$-ideal in $\Lop(X_c)$. By \cite[Theorem~1]{HanKim2025} this yields the adjoint
compact perturbation property. Since $X_c$ is reflexive, the latter is equivalent to the
CPP for $(X_c,X_c)$; compare also \cite{Werner1992}. We have preferred the duality
argument above because it also makes transparent the logical relation between the two
questions considered in the paper.
\end{remark}

\begin{remark}
The construction exhibits a geometric asymmetry which is independent of the compact
perturbation argument. Both $X_c$ and $X_c^*$ have property $(M)$, but $X_c$ fails the
Opial property by Remark~\ref{rem:XnotOpial}, whereas $X_c^*$ satisfies the stronger
quantitative estimate \eqref{eq:quantOpial}. In the terminology of Han and Kim, this is
the strictness needed to recover the WMP on the dual side.
\end{remark}

\section*{Funding}
The first and second authors have been supported by the grant PID2021-122126NB-C33 funded by MICIU / AEI / 10.13039 / 501100011033. The first author has also been supported by the grant PID2025-168101NB-I00 funded by MICIU / AEI / 10.13039 / 501100011033.

\end{document}